\documentclass[12pt]{amsart} 
\usepackage{amssymb, graphicx,epic,eepic}
\usepackage{color, graphicx} 

\theoremstyle{plain} 
\newtheorem{theorem}{\indent\sc \bf Theorem}[section] 
\newtheorem{lemma}[theorem]{\indent\sc \bf Lemma}
\newtheorem{corollary}[theorem]{\indent\sc \bf Corollary}
\newtheorem{proposition}[theorem]{\indent\sc \bf Proposition}

\theoremstyle{definition} 
\newtheorem{definition}[theorem]{\indent\sc \bf Definition}
\newtheorem{remark}[theorem]{\indent\sc \bf Remark}

\begin{document}

\title[Prolongations of regular overdetermined systems]
{On prolongations of second-order regular overdetermined systems 
with two independent and one dependent variables.}

\author[T. Noda]{Takahiro Noda}

\renewcommand{\thefootnote}{\fnsymbol{footnote}}
\footnote[0]{2010\textit{ Mathematics Subject Classification}.
Primary 58A15; Secondary 58A17.}


\address{
Takahiro Noda, \endgraf
Graduate School of Mathematics, 
Nagoya University, \endgraf
Chikusa-ku, Nagoya 464-8602,  
Japan.}
\email{m04031x@math.nagoya-u.ac.jp}

\maketitle


\begin{abstract}
The purpose of this present paper is to investigate the geometric structure 
of regular overdetermined systems of second order with 
two independent and one dependent variables from the point of 
view of rank 2 prolongations. 
Utilizing this notion of prolongations, we characterize 
the type of these overdetermined systems. We also give a 
systematic method for constructing the geometric singular solutions 
by analyzing of a decomposition of this prolongation. 
As an applications, 
we determine the geometric singular solutions of Cartan's 
overdetermined systems.  
\end{abstract}

\section{Introduction}
The subject of this paper is second-order 
regular overdetermined systems  
with 2 independent and 1 dependent variables. 
For these overdetermined systems, various 
pioneering works are given by many researchers 
(cf. \cite{C1}, \cite{L}, \cite{KL1}, \cite{Y1}, \cite{Y6}). 
In particular, the study of overdetermined involutive systems 
has accomplished many significant results. 
First E. Cartan \cite{C1} characterized overdetermined 
involutive systems by the condition that these admit 
a one-dimensional Cauchy characteristic system. 
He also found out the systematic method for constructing  
regular solutions of involutive systems. The 
precise definition of regular solutions is given in 
Definition \ref{solution} of this paper.  
Recently, these consideration are 
reformulated as the theory of PD-manifolds by Yamaguchi 
(cf. \cite{Y1}, \cite{Y6}). 
In addition, Kakie (cf. \cite{K1}, \cite{K2}) studied 
involutive systems 
including the existence of regular solutions in $C^{\infty}$-category 
and Cauchy problems by using the theory of characteristic systems.\par   
In this paper, we investigate 
regular overdetermined systems from the point of view of  
rank 2 prolongations. Now we introduce the notion of regular 
overdetermined systems. 
Let $J^{2}(\mathbb R^2, \mathbb R)$ be the 2-jet space: 
\begin{equation}
J^{2}(\mathbb R^2, \mathbb R):=\left\{(x,y,z,p,q,r,s,t)\right\}
\end{equation}
This space has the canonical system 
$C^2=\left\{\varpi_{0}=\varpi_{1}=\varpi_{2}=0\right\}$ given by the 
annihilators: 
\begin{equation*}
\varpi_{0}:=dz-pdx-qdy,\quad 
\varpi_{1}:=dp-rdx-sdy,\quad 
\varpi_{2}:=dq-sdx-tdy.
\end{equation*}
This jet space is also constructed geometrically as the 
Lagrange-Grassmann bundle over the standard contact 
5-dimensional manifold. For more details, see \cite{Y6}.  
On the 2-jet space, we consider overdetermined systems of the form: 
\begin{equation}\label{equation}
F(x,y,z,p,q,r,s,t)=G(x,y,z,p,q,r,s,t)=0, 
\end{equation}
where $F$ and $G$ are smooth functions on 
$J^{2}(\mathbb R^2, \mathbb R)$. 
We set 
$R=\left\{F=G=0 \right\} \subset J^2(\mathbb R^2, \mathbb R)$  
and restrict the canonical differential system $C^2$ to $R$. We denote 
it by $D(:=C^2 |_R)$. 
An overdetermined system $R$ is called {\it regular} if 
two vectors $(F_r,F_s,F_t)$ and $(G_r,G_s,G_t)$ are linearly independent 
on $R$. 
Then, $R$ is a submanifold of codimension 2, and the restriction
$\pi^{2}_{1}|_{R}:R \to J^1(\mathbb R^2, \mathbb R)$ of the natural  
projection $\pi^{2}_{1}:J^2(\mathbb R^2, \mathbb R) \to J^1(\mathbb R^2, \mathbb R)$ is a submersion. Due to the property, restricted 1-forms $\varpi_{i}|_{R}$ on $R$ are linearly independent. Hence  
$D=\left\{\varpi_{0}|_{R}=\varpi_{1}|_{R}=\varpi_{2}|_{R}=0\right\}$ 
is a rank 3 system on $R$. For brevity, 
we denote each restricted generator 1-form $\varpi_{i}|_{R}$ 
of $D$ by $\varpi_{i}$ in the following. 
Regular overdetermined systems or associated differential systems are 
classified into 4 types consisting of involutive-type, two finite-types, and 
torsion-type under the structure equations or the 
symbol algebras (see, section 3). 
We thus investigate regular overdetermined systems 
for each case by using the theory of rank 2 prolongations.
Here, the notion of the rank 2 prolongations can be regraded as 
a generalization of prolongations with the transversality condition 
and some researchers introduced this notion in each 
way (cf. \cite{KL1}, \cite{NS2}). 
In particular, this notion can be regarded as a higher-rank 
version of Cartan prolongation which is investigated  
by Montgomery and Zhitomirskii \cite{MZ}.     
The notion of  prolongations with the transversality condition 
is studied deeply and their structures are well-known 
(cf. \cite{C1}, \cite{KL2}, \cite{KL3}, \cite{Y1}, \cite{Y6}). 
On the other hand, there are unexplored territories 
for these rank 2 prolongations.   
Thus, one of our main motivation in this paper is to reveal the difference 
between prolongations with the transversality condition and 
rank 2 prolongations. For more details, see section 3.  
\par  
The paper is organized as follows:  
In section 2, we prepare some terminology and notation 
for the study of differential systems. 
In section 3, we define the rank 2 prolongations of differential systems 
and determine the topology of fibers of regular overdetermined 
systems in terms of this notion 
({\bf Theorem \ref{characterization-prolongation}}). 
As a direct consequence followed by this characterization, 
we obtain the specific difference between prolongations with the transversality condition and rank 2 prolongations ({\bf Corollary \ref{prolongation-difference}}). 
From this characterization, we get a new point of view of  
the geometric structure of overdetermined systems.  
In section 4, we study the structures of the 
canonical systems $\hat D$ on the rank 2 prolongations 
$\Sigma(R)$ of (locally) involutive systems. More precisely, we  
clarify the structure of nilpotent graded Lie algebras (symbol algebras) 
of the canonical systems on the rank 2 prolongations by using  
some decomposition ({\bf Proposition \ref{symbol}}). 
Here, it is well-known that 
the symbol algebras are fundamental invariants of 
differential systems or filtered manifolds.       
We also have the tower structure of these involutive systems 
by the successive prolongations (see {\bf Theorem \ref{tower}}).  
In section 5, we provide two systematic approach to 
construct the geometric singular solutions of involutive systems. 
Moreover, we apply these methods to 
Cartan's overdetermined system. For this system, E. Cartan \cite{C1} 
already gave the explicit integral representation 
of the regular solutions. On the other hand, the explicit description of 
singular solutions has not been given yet. Thus,  
we give the explicit integral representation of  
geometric singular solutions of this system.\par       

\section{Regularity and Symbol algebra of differential systems}
In this section, we prepare some terminology and notation for the study of differential systems. 
For more details, we refer the reader to \cite{T1} and \cite{Y3}.    

\subsection{Derived system, Weak derived system} 
Let $D$ be a differential system on a manifold $R$. 
We denote by $\mathcal D=\Gamma(D)$ the sheaf of sections to $D$. 
The derived system $\partial D$ of a differential system $D$ is defined, 
in terms of sections, by 
$\partial \mathcal D:=\mathcal D+[\mathcal D,\mathcal D].$ 
In general, $\partial D$ is obtained as a subsheaf of the tangent sheaf of $R$. 
Moreover, higher derived systems $\partial^{k} D$ are defined 
successively by 
$\partial^{k} \mathcal D:=\partial (\partial^{k-1} \mathcal D),$ 
where we set $\partial^{0} D=D$ by convention. On the other hand, 
the $k$-th weak derived systems 
$\partial^{(k)} D$ of $D$ are defined inductively by 
$\partial^{(k)} \mathcal D:=\partial^{(k-1)}\mathcal D
                             +[\mathcal D,\partial^{(k-1)}\mathcal D].$ 

\begin{definition}
A differential system $D$ is called regular (resp. weakly regular) 
if $\partial^{k} D$ 
(resp. $\partial^{(k)} D$) is a subbundle for each $k$. 
\end{definition}

If $D$ is not weakly regular around $x\in R$, then $x$ is called 
a {\it singular point} in the sense 
of Tanaka theory. These derived systems are also interpreted by using annihilators as follows \cite{Shi}:  
Let $D=\{\varpi_1=\cdots =\varpi_s=0 \}$ be a differential system on $R$. 
We denote by $D^{\perp}$ the annihilator 
subbundle of $D$ in $T^* R$, that is, 
\begin{eqnarray*}
D^{\perp}(x) &:=& 
\{\omega\in T_{x}^{*} R\ |\ \omega(X)=0\ \text{for any}\ X \in D(x)
 \}, \\
 &=& <\varpi_1, \cdots ,\varpi_s >.
\end{eqnarray*}
Then the annihilator $(\partial D)^{\perp}$ of the first derived system of $D$ 
is given by 
\[
(\partial D)^{\perp}=\{\varpi \in D^{\perp}\ |\ d\varpi \equiv 0\ 
(\bmod\ D^{\perp})\}.
\]
Moreover the annihilator $(\partial^{(k+1)} D)^{\perp}$ of the $(k+1)$-th weak
derived system of $D$ is given by  
\begin{eqnarray*}
(\partial^{(k+1)} D)^{\perp} &=& \{\varpi \in (\partial^{(k)} D)^{\perp}\ |\ 
d\varpi \equiv 0\ (\bmod\ (\partial^{(k)} D)^{\perp}, \\
& & \hspace{3cm} (\partial^{(p)} D)^{\perp}\wedge (\partial^{(q)} D)^{\perp},\ 2 \le
p,q \le k-1)\}.
\end{eqnarray*}
We set $D^{-1}:=D,\ D^{-k}:=\partial ^{(k-1)} D$ ($k\geq 2$), for 
a weakly regular differential system $D$. Then we have 
(\cite[Proposition 1.1]{T1}) 
\begin{enumerate} 
\item [(T1)] There exists a unique positive integer $\mu$ such that 
$$D^{-1}\subset D^{-2}\subset \cdot\cdot\cdot \subset D^{-k} \subset \cdot\cdot\cdot 
\subset D^{-(\mu-1)} \subset D^{-\mu}=D^{-(\mu+1)}=\cdot\cdot\cdot$$
\item[(T2)] $[\mathcal D^{p}, \mathcal D^{q}]\subset \mathcal D^{p+q}$\quad \quad  
for\ all\ $p,\ q <0$. 
\end{enumerate}

\subsection{Symbol algebra of differential system}
Let $(R, D)$ be a weakly regular 
differential system such that 
$$TR=D^{-\mu}\supset D^{-(\mu-1)}\supset \cdot\cdot\cdot \supset
D^{-1}=:D.$$
For all $x\in R$, we set $\mathfrak g_{-1}(x):=D^{-1}(x)=D(x),\    
\mathfrak g_{p}(x):=D^{p}(x)/D^{p+1}(x),\ (p=-2,-3, \ldots, -\mu)$ and 
$$\mathfrak m(x):=\bigoplus_{p=-1}^{-\mu} \mathfrak g_{p}(x).$$
Then dim $\mathfrak m(x)=$ dim $R$ holds.  
We set $\mathfrak g_{p}(x)=\left\{0\right\}$ when $p\leq-\mu-1$. 
For $X\in \mathfrak g_{p}(x),\ Y\in \mathfrak g_{q}(x)$, the Lie bracket 
$[X, Y]\in\mathfrak g_{p+q}(x)$ is defined as follows: 
Let $\tilde X\in \mathcal D^{p},\ \tilde Y\in \mathcal D^{q}$ 
be extensions 
($\tilde X_{x}=X,\ \tilde Y_{x}=Y$). Then $[\tilde X, \tilde Y]\in \mathcal D^{p+q}$,  
and we set $[X,Y]:=[\tilde X, \tilde Y]_{x}\in \mathfrak g_{p+q}(x)$. 
It does not depend on the choice of the extensions  
because of the equation
\[
[f \tilde{X}, g \tilde{Y}] =fg[\tilde{X},\tilde{Y}] 
+f(\tilde{X} g) \tilde{Y} -g(\tilde{Y} f)\tilde{X}\quad 
(f,g \in C^{\infty}(R)).
\]
 The Lie algebra 
$\mathfrak m(x)$ is a nilpotent graded Lie algebra. 
We call $(\mathfrak m(x),\ [\ ,\ ])$ 
the {\it symbol algebra} of $(R, D)$ at $x$. 
Note that the symbol algebra $(\mathfrak m(x),\ [\ ,\ ])$ satisfies the generating conditions 
$$
[\mathfrak g_{p}, \mathfrak g_{-1}]=\mathfrak g_{p-1}\ \ (p<0). 
$$ 
Later, Morimoto \cite{M} introduced the notion of a filtered manifold as generalization 
of the weakly regular differential system. 
 
We define a filtered manifold $(R,F)$ by a pair of a manifold $R$ 
and a tangential filtration $F$. Here, a tangential filtration 
$F$ on $R$ is a sequence $\left\{F^{p}\right\}_{p<0}$ of subbundles of 
the tangent bundle $TR$ and the following conditions are satisfied: 
\begin{enumerate} 
\item [(M1)] $TR=F^{k}=\cdot\cdot\cdot =F^{-\mu}\supset \cdot\cdot\cdot \supset 
F^{p}\supset F^{p+1}\supset \cdot\cdot\cdot \supset F^{0}=\left\{0\right\},$
\item[(M2)] $[\mathcal F^{p}, \mathcal F^{q}]\subset \mathcal F^{p+q}$\quad \quad  
for\ all\ $p,\ q<0$,   
\end{enumerate}
where $\mathcal F^{p}=\Gamma(F^{p})$ is the set of sections of $F^{p}$. 
Let $(R,F)$  be a filtered manifold. For $x\in R$ we set 
$\mathfrak f_{p}(x):=F^{p}(x)/F^{p+1}(x)$,  
and 
$$\mathfrak f(x):=\bigoplus_{p<0} \mathfrak f_{p}(x).$$
For $X \in \mathfrak f_{p}(x),\ Y \in \mathfrak f_{q}(x)$, 
the Lie bracket $[X,Y]\in \mathfrak f_{p+q}(x)$ is defined as follows: 
Let $\tilde X\in \mathcal F^{p},\ \tilde Y\in \mathcal F^{q}$ be extensions 
($\tilde X_{x}=X,\ \tilde Y_{x}=Y$). Then $[\tilde X, \tilde Y]\in \mathcal F^{p+q}$,  
and we set $[X,Y]:=[\tilde X, \tilde Y]_{x}\in \mathfrak f_{p+q}(x)$. 
It does not depend on the choice of the extensions. 
The Lie algebra 
$\mathfrak f(x)$ is also a nilpotent graded Lie algebra. 
We call $(\mathfrak f(x),\ [\ ,\ ])$ the {\it symbol algebra} of 
$(R, F)$ at $x$. 
In general it does not satisfy the generating conditions.

\section{Rank $2$ prolongations of regular overdetermined systems} 
In this section, we provide the rank $2$ prolongations for regular overdetermined systems of  
second-order of codimension 2 with $2$ independent and $1$ dependent variables. First, we introduce the notion of the rank 2 prolongations 
of differential systems, in general. 
\begin{definition}\label{rank 2-prolongation}
Let $(R,D)$ be a differential system. Then, the {\it rank $2$ prolongation} 
of $(R,D)$ is defined by 
\begin{equation}\label{prolongation}
\Sigma(R):=\bigcup_{x\in R}\Sigma_{x}, 
\end{equation}
where $\Sigma_x=\left\{v \subset T_{x}R\ |\ v\ {\rm is\ a}\ 2{\rm \textit{-}dim.\ integral\ element\ of\ } 
D(x)\quad ({\rm i.e.}\ d\varpi_{i}|_{v}=0)\ \right\}.$ 
Let $p:\Sigma(R)\to R$ be the projection. We define the canonical system 
$\hat D$ on $\Sigma(R)$ by 
\begin{align}\label{canonical system}
\hat D(u):&={p_{*}^{-1}}(u),\\ 
          &=\left\{v\in T_{u}(\Sigma(R))\ |\ p_{*}(v)\in u\right\}, \nonumber
\end{align}
where $u\in \Sigma(R).$ 
\end{definition}
This space $\Sigma(R)$ is a subset of the following Grassmann bundle over $R$
\begin{equation}\label{Grassmann}
J(D, 2):=\bigcup_{x\in R} J_x
\end{equation}
where 
$J_x:=\left\{v \subset T_{x}R\ |\ v\ {\rm is\ a}\ 2{\rm \textit{-}dim.\ subspace\ of}\ D(x)\right\}.$  
In general, the rank $2$ prolongations $\Sigma(R)$ have singular points, that is $\Sigma(R)$ is not smooth. 
In fact, for prolongations, there also exists the notion of 
{\it prolongations with the transversality condition} : 
\begin{equation}\label{prolongation-independence}
R^{(1)}=\bigcup_{x\in R}\ R^{(1)}_{x},
\end{equation}
where 
$R^{(1)}_x=\left\{2{\rm -dim.\ integral\ elements\ 
of}\ D(x),\ {\rm transversal\ to}\  Ker {(\pi^{2}_{1}|_{R}})_{*}  \right\}.$ 
For this notion, it is well-known the geometric structures by many 
workers (\cite{BCG3}, \cite{C1}, \cite{Y1}). 
In this paper, we examine the prolongations except for this 
transversality condition.\par
Next, we explain a classification of the type of 
overdetermined systems under the structure equations or 
the corresponding symbols. 
Let $(R,D)$ be a regular overdetermined system. 
If $(R,D)$ does not have torsion, that is,  
the prolongation $R^{(1)}$ is onto,  
then the structure equation of this system is one of 
the following three cases 
(\cite[the case of codim $\mathfrak f=2$ of Case $n=2$ in 
p. 346--347]{Y6}): 
\begin{enumerate}
\item[(I)] There exists a coframe 
$\left\{\varpi_0, \varpi_1,\varpi_2,\omega_1,\omega_2, 
\pi\right\}$ around $w\in R$ such that   
$D=\left\{\varpi_0=\varpi_1=\varpi_2=0\right\}$ 
and the following structure equation holds at $w$: 
\begin{align}
d\varpi_{0}&\equiv \omega_1 \wedge \varpi_1+\omega_2 \wedge \varpi_2\quad \quad  
{\rm mod\ } \varpi_0, \nonumber \\
d\varpi_{1}&\equiv \ 0 \quad \quad \quad \quad \quad \quad \quad  
{\quad \quad \rm mod\ } 
\varpi_{0}, \varpi_{1}, \varpi_{2}, \label{involutive-str-equ} \\ 
d\varpi_{2}&\equiv \quad \quad \quad \quad \quad \omega_2 \wedge \pi  
{\quad \quad \rm mod\ } 
\varpi_{0}, \varpi_{1}, \varpi_{2}. \nonumber
\end{align}
\item[(II)] There exists a coframe 
$\left\{\varpi_0, \varpi_1,\varpi_2,\omega_1,\omega_2, 
\pi \right\}$ around $w\in R$ such that   
$D=\left\{\varpi_0=\varpi_1=\varpi_2=0\right\}$
and the following structure equation holds at $w$:   
\begin{align}
d\varpi_{0}&\equiv \omega_{1}\wedge \varpi_{1}+\omega_{2}\wedge \varpi_{2}
\quad \quad {\rm mod}\ \varpi_{0}, \nonumber \\
d\varpi_{1}&\equiv \quad \quad \quad \quad \quad \omega_{2}\wedge \pi 
{\quad \quad \rm mod\ } 
\varpi_{0}, \varpi_{1}, \varpi_{2},\label{par1-str-equ}\\ 
d\varpi_{2}&\equiv \omega_{1}\wedge \pi \quad \quad \quad
{\quad \quad \quad \quad \rm mod\ } 
\varpi_{0}, \varpi_{1}, \varpi_{2}. \nonumber
\end{align}
\item[(III)] There exists a coframe 
$\left\{\varpi_0, \varpi_1,\varpi_2,\omega_1,\omega_2, 
\pi \right\}$ around $w\in R$ such that   
$D=\left\{\varpi_0=\varpi_1=\varpi_2=0\right\}$ 
and the following structure equation holds at $w$: 
\begin{align}
d\varpi_{0}&\equiv \omega_{1}\wedge \varpi_{1}+\omega_{2}\wedge \varpi_{2}
\quad \quad {\rm mod}\ \varpi_{0}, \nonumber \\
d\varpi_{1}&\equiv \omega_{1}\wedge \pi \ \quad \quad \quad \quad \quad  
{\quad \quad \rm mod\ } 
\varpi_{0}, \varpi_{1}, \varpi_{2}, \label{par2-str-equ}\\ 
d\varpi_{2}&\equiv \quad \quad \quad \quad \quad 
 \omega_{2}\wedge \pi 
{\quad \quad \ \rm mod\ } 
\varpi_{0}, \varpi_{1}, \varpi_{2}, \nonumber
\end{align}
\end{enumerate}
Now we consider the case where torsion exists, that is,  
$R^{(1)}$ is not onto.  
In fact, then the structure equation (or the symbol) of torsion type 
has the normal form by the obtained result in \cite{NSY}.  
This fact also follows from the technique of the proof of  
\cite[Theorem 3.3]{NS1}.    
Namely, if $(R,D)$ has torsion at $w\in R$, 
we have the following structure equation at $w$.\par   
(IV) There exists a coframe 
$\left\{\varpi_0, \varpi_1,\varpi_2,\omega_1,\omega_2, 
\pi \right\}$ around $w\in R$ such that   
$D=\left\{\varpi_0=\varpi_1=\varpi_2=0\right\}$ 
and the following structure equation holds at $w$:
\begin{align}
d\varpi_{0}&\equiv \omega_1 \wedge \varpi_1+\omega_2 \wedge \varpi_2\quad \quad  
{\rm mod\ } \varpi_0, \nonumber \\
d\varpi_{1}&\equiv \omega_1 \wedge \omega_2 \quad \quad \quad \quad \ \ 
{\quad \quad \rm mod\ } 
\varpi_{0}, \varpi_{1}, \varpi_{2}, \label{torsion-str-equ}\\ 
d\varpi_{2}&\equiv \quad \quad \quad \quad \quad \omega_2 \wedge \pi  
{\quad \quad \rm mod\ } 
\varpi_{0}, \varpi_{1}, \varpi_{2}. \nonumber
\end{align}
Here the types of (I), (II), (III) and (IV) correspond to 
differential systems of involutive type, two finite types,
and torsion type respectively (cf. \cite{Y5}, \cite{Y6}). 
From now on, We often call these systems the 
differential systems of type ($k$), where $k=$ I, II, III, IV.\par 
\begin{remark}
The structures of prolongations $R^{(1)}$ with the 
transversality condition for  
$R$ of types of  (I), (II), (III) and (IV) are known. Indeed,  
$R^{(1)}\to R$ is a $\mathbb R$-bundle for the type of  (I). 
Moreover $R^{(1)}$ is diffeomorphic to $R$ for types of (II) or (III), and 
the set $R^{(1)}$ is empty for types of (IV). 
\end{remark}     
One of the main purpose of this section is to 
clarify the difference between $R^{(1)}$ and 
$\Sigma(R)$. First, we consider the case of type (I). 
\begin{lemma}\label{involutive-prolongation}
Let $(R,D)$ be a differential system of type $(I)$ with $2$ independent 
and $1$ dependent variables. 
Then the rank $2$ prolongation $\Sigma(R)$ is a 
smooth submanifold of $J(D,2)$. 
Moreover, it is a $S^1$-bundle over $R$.  
\end{lemma}
\begin{proof}
Let $\pi:J(D,2)\to R$ be the projection and $U$ an open set in $R$. 
Then $\pi^{-1}(U)$ is covered by 3 open sets in $J(D,2)$, that is,   
\begin{equation}\label{Grassmann-cover}
\pi^{-1}(U)=U_{\omega_{1}\omega_{2}}\cup U_{\omega_{1}\pi}
\cup U_{\omega_{2}\pi},
\end{equation}
where  
\begin{align*} 
&U_{\omega_{1}\omega_{2}}:=\left\{v \in \pi^{-1}(U) \ |\ \omega_{1}|_{v}\wedge \omega_{2}|_{v} \not=0\right\},\ 
U_{\omega_{1}\pi}:=\left\{v \in \pi^{-1}(U) \ |\ \omega_{1}|_{v}\wedge \pi|_{v} \not=0\right\},\\
&U_{\omega_{2}\pi}:=\left\{v \in \pi^{-1}(U) \ |\ \omega_{2}|_{v}\wedge \pi|_{v} \not=0\right\}. 
\end{align*}
We explicitly describe the defining equation of $\Sigma(R)$ 
in terms of the inhomogeneous Grassmann coordinate of fibers in 
$U_{\omega_{1}\omega_{2}}, U_{\omega_{1}\pi}, U_{\omega_{2}\pi}$.  
First we consider it on $U_{\omega_{1}\omega_{2}}$. 
For $w\in U_{\omega_{1}\omega_{2}}$, $w$ is a $2$-dimensional subspace of $D(v)$, where $p(w)=v$. Hence, by restricting $\pi$ to $w$, 
we can introduce the 
inhomogeneous coordinate $p^{1}_{i}$ $(i=1, 2)$ 
of fibers of $J(D,2)$ around $w$ with 
$\pi|_{w}={p^{1}_{1}}(w)\omega_{1}|_{w}+{p^{1}_{2}}(w)\omega_{2}|_{w}$. 
Moreover, $w$ satisfies $d\varpi_{2}|_{w}\equiv 0$
in (\ref{involutive-str-equ}). Hence, we show that 
\begin{align*}
d\varpi_{2}|_{w}\equiv \omega_{2}|_{w}\wedge \pi|_{w}
                &\equiv {p_{1}^{1}}(w)\omega_{2}|_{w}\wedge\omega_{1}|_{w}.
\end{align*} 
Thus we obtain the defining equations $f=0$ of $\Sigma(R)$ in $U_{\omega_{1}\omega_{2}}$ 
of $J(D,2)$, where $f=p_{1}^{1}$, that is, $\left\{f=0 \right\}\subset U_{\omega_{1}\omega_{2}}.$ 
Then $df$ does not vanish on $\left\{f=0\right\}$. 
In the same way, on $U_{{\omega}_{1}\pi}$, $df$ does not vanish on $\Sigma(R)$. 
Finally we consider on $U_{\omega_{2}\pi}$.  
Then an element $w\in U_{\omega_{2}\pi}$ is a 
$2$-dimensional subspace of $D(v)$, 
where $p(w)=v$. Hence, by restricting $\omega_1$ to $w$, 
we can introduce the 
inhomogeneous coordinate $p^{3}_{i}$ (i=1, 2) 
of fibers of $J(D,2)$ around $w$ with 
$\omega_{1}|_w={p^{3}_{1}}(w)\omega_{2}|_{w}+{p^{3}_{2}}(w)\pi|_{w}$. 
Moreover, $w$ satisfies $d\varpi_{2}|_{w}\equiv 0$. 
However we have 
$$ 
d\varpi_{2}|_{w}\equiv \omega_{2}|_{w} \wedge \pi|_{w}\not=0. 
$$  
Thus, there does not exist an integral element, 
that is, $U_{\omega_{2}\pi}\cap p^{-1}(U)=\emptyset$. 
Therefore, we conclude $\Sigma(R)$ is a 
submanifold in $J(D,2)$.\par 
Next, we show that the topology of its fibers is $S^1$. 
For any open set $U\subset R$, we obtain the covering      
\begin{equation}\label{rank2-cover}
p^{-1}(U)=U_{\omega_{1}\omega_{2}}\cup U_{\omega_{1}\pi}. 
\end{equation}
Then the canonical system $\hat D$ of rank $3$ is given by   
$$
\hat D=\left\{\varpi_{0}=\varpi_{1}=\varpi_{2}=\varpi_{\pi}=0\right\}\quad \text{on $U_{\omega_{1}\omega_{2}}$}, 
$$
where $\varpi_{\pi}=\pi-a \omega_{2}$ and 
$a$ is fiber coordinate and 
$$
\hat D=\left\{\varpi_{0}=\varpi_{1}=\varpi_{2}=\varpi_{\omega_{2}}=0\right\}\quad \text{on $U_{\omega_{1}\pi}$}, 
$$
where $\varpi_{\omega_{2}}=\omega_{2}-b \pi$ and $b$ is a fiber coordinate.   
To prove the statement, we consider the gluing of $(\Sigma(R), \hat D)$ . 
Let $w\in p^{-1}(U)$ be a point in $U_{\omega_{1}\pi}\subset p^{-1}(U)$. 
Here, if $w\not \in U_{\omega_{1}\omega_{2}}$, then we have 
$b=0$ because of the condition $\omega_{1}\wedge \omega_{2}=0$. 
Thus we show that $b\not= 0$ on $U_{\omega_{1}\omega_{2}}\cap U_{\omega_{1}\pi}$  
and $b=0$ on $U_{\omega_{1}\pi} \backslash U_{\omega_{1}\omega_{2}}$. 
Hence we can prove that 
the topology of fibers is $S^1$. Indeed, we obtain the  
transition function $\phi$ of $(\Sigma(R), \hat D)$ 
on $U_{\omega_{1}\omega_{2}}\cap U_{\omega_{1}\pi}$ defined by  
\begin{equation*}
\phi (v, a)=\left(v, b:=\frac{1}{a} \right)\quad \ {\rm for}\ a\not= 0,
\end{equation*}
where $v$ is a local coordinate on $R$. 
This map $\phi$ satisfies the condition $\phi_{*}\hat D=\hat D$ on 
$U_{\omega_{1}\omega_{2}}\cap U_{\omega_{1}\pi}$. 
Thus $\phi$ is the  projective transformation of 
$\mathbb {RP}^{1}\cong S^1$. 
Note that each fiber in codimension 1 submanifold 
$\left\{b=0\right\}\subset U_{\omega_{1}\pi}$ corresponds to 
the point at infinity of $\mathbb {RP}^{1}$.   
\end{proof}
Next, we consider the case of type (II).
\begin{lemma}\label{par1-prolongation}
Let $(R,D)$ be a differential system of type $(${\rm II}$)$ 
with $2$ independent and $1$ dependent 
variables. Then the rank $2$ prolongation $\Sigma(R)$ is 
diffeomorphic to $R$.    
\end{lemma}
\begin{remark}\label{remark-par1-prolongation}
We emphasize that $(R,D)$ and $(\Sigma(R), \hat D)$ are different 
as differential systems. Indeed $D$ is a rank $3$ differential system on $R$, 
but $\hat D$ is a rank $2$ differential system on $\Sigma(R)$. 
\end{remark}
\begin{proof}
In this situation we also use the covering  (\ref{Grassmann-cover}) of $\pi^{-1}(U)$ 
for the Grassmann bundle $J(D,2)$ and explicitly  
describe the defining equation of $\Sigma(R)$ 
in terms of the inhomogeneous Grassmann coordinate of fibers in 
$U_{\omega_{1}\omega_{2}}, U_{\omega_{1}\pi}, U_{\omega_{2}\pi }$. 
First we consider it on $U_{\omega_{1}\omega_{2}}$. 
For $w\in U_{\omega_{1}\omega_{2}}$, $w$ is a $2$-dimensional subspace of $D(v)$, where $p(w)=v$. 
Hence, by restricting $\pi$ to $w$, we can introduce the 
inhomogeneous coordinate $p^{1}_{i}$ of fibers of 
$J(D,2)$ around $w$ with $\pi |_{w}={p^{1}_{1}}(w)\omega_{1}|_{w}+{p^{1}_{2}}(w)\omega_{2}|_{w}$. 
Moreover $w$ satisfies 
$d\varpi_{1}|_{w}\equiv d\varpi_{2}|_{w}\equiv 0$
in (\ref{par1-str-equ}). Thus we get  
\begin{align*}
d\varpi_{1}|_{w}&\equiv \omega_{2}|_{w}\wedge \pi|_{w} 
                \equiv p^{1}_{1}(w)\omega_{2}|_{w}\wedge\omega_{1}|_{w},\\
d\varpi_{2}|_{w}&\equiv \omega_{1}|_{w}\wedge \pi|_{w}
                \equiv p^{1}_{2}(w)\omega_{1}|_{w}\wedge\omega_{2}|_{w}. 
\end{align*} 
In this way, we obtain the defining equations $f=g=0$ of $\Sigma(R)$ in  $U_{\omega_{1}\omega_{2}}$ of $J(D,2)$, where 
$f=p_{1}^{1}, g=p_{2}^{1}$. Hence we have one trivial integral element. 
Next we consider on $U_{\omega_{1}\pi}$. 
In the same way, by restricting $\omega_{2}$ to $w$, 
we can introduce the  
inhomogeneous coordinate $p^{2}_{i}$ of fibers of $J(D,2)$ 
around $w$ with $\omega_{2}|_{w}={p^{2}_{1}}(w)\omega_{1}|_{w}+{p^{2}_{2}}(w)
\pi |_{w}$.  
Moreover $w$ satisfies 
$d\varpi_{1}|_{w}\equiv d\varpi_{2}|_{w}\equiv 0$. However we have  
$d\varpi_{2}|_{w} \equiv \omega_{1}|_{w}\wedge \pi |_{w} \not = 0.$  
Hence there does not exist an integral element. 
Finally we consider on $U_{\omega_{2}\pi }$. 
In this situation, by restricting $\omega_1$ to $w$, we can also 
introduce the inhomogeneous coordinate $p^{3}_{i}$ of fibers of $J(D,2)$ around $w$ with 
$\omega_{1}|_w={p^{3}_{1}}(w)\omega_{2}|_{w}+{p^{3}_{2}}(w)\pi |_{w}.$  
Moreover $w$ satisfies 
$d\varpi_{1}|_{w}\equiv d\varpi_{2}|_{w}\equiv 0$. 
However, we have $d\varpi_{1}|_{w}\equiv \omega_{2}|_{w} \wedge \pi |_{w} \not=0.$     
Hence there does not exist an integral element. 
Therefore, $\Sigma(R)$ is a section of the 
Grassmann bundle $J(D,2)$ over $R$.   
\end{proof}

Next we consider the case of type (III). 

\begin{lemma}\label{par2-prolongation}
Let $(R,D)$ be a differential system of type $(${\rm III}$)$ 
with $2$ independent and $1$ dependent variables. 
Then the rank $2$ prolongation $\Sigma(R)$ is equal to $R$.   
\end{lemma}

\begin{remark}\label{remark-par2-prolongation}
Two differential systems $(R,D)$ and $(\Sigma(R), \hat D)$ are different 
in the same way to Remark \ref{remark-par1-prolongation}.    
\end{remark}

\begin{proof}
We also use the covering  (\ref{Grassmann-cover}) of $\pi^{-1}(U)$ 
for the Grassmann bundle $J(D,2)$ and explicitly describe the defining equation of $\Sigma(R)$ 
in terms of the inhomogeneous Grassmann coordinate of fibers in 
$U_{\omega_{1}\omega_{2}}, U_{\omega_{1}\pi}, U_{\omega_{2}\pi }$. 
First we consider it on $U_{\omega_{1}\omega_{2}}$. 
For $w\in U_{\omega_{1}\omega_{2}}$, 
$w$ is a $2$-dimensional subspace of $D(v)$, 
where $p(w)=v$. 
Hence by restricting $\pi$ to $w$, we can introduce the 
inhomogeneous coordinate $p^{1}_{i}$ of fibers of $J(D,2)$ around $w$ 
with $\pi |_{w}={p^{1}_{1}}(w)\omega_{1}|_{w}+{p^{1}_{2}}(w)\omega_{2}|_{w}$. 
Moreover $w$ satisfies $d\varpi_{1}|_{w}\equiv d\varpi_{2}|_{w}\equiv 0$ in (\ref{par2-str-equ}). Thus we have  
\begin{align*}
d\varpi_{1}|_{w}&\equiv \omega_{1}|_{w}\wedge \pi|_{w}
                \equiv p^{1}_{2}(w)\omega_{1}|_{w}\wedge\omega_{2}|_{w},\\ 
d\varpi_{2}|_{w}&\equiv \omega_{2}|_{w}\wedge \pi|_{w}
                \equiv p^{1}_{1}(w)\omega_{2}|_{w}\wedge\omega_{1}|_{w}. 
\end{align*} 
In this way, we obtain the defining equations $f=g=0$ of $\Sigma(R)$ 
in $U_{\omega_{1}\omega_{2}}$ of $J(D,2)$, where 
$f=p_{2}^{1}, g=p_{1}^{1}$. Hence, we have one trivial integral element. 
Next we consider on $U_{\omega_{1}\pi }$. 
For the same way, by restricting $\omega_{2}$ to $w$, we can introduce the 
inhomogeneous coordinate $p^{2}_{i}$ of fibers of $J(D,2)$ around $w$ 
with $\omega_{2}|_{w}={p^{2}_{1}}(w)\omega_{1}|_{w}+{p^{2}_{2}}(w)\pi |_{w}$. 
Moreover $w$ satisfies $d\varpi_{1}|_{w}\equiv d\varpi_{2}|_{w}\equiv 0$. However we have 
$d\varpi_{1}|_{w} \equiv \omega_{1}|_{w}\wedge \pi |_{w} \not = 0$. 
Hence there does not exist an integral element. 
Finally we consider on $U_{\omega_{2}\pi }$. 
In this situation, by restricting $\omega_1$ to $w$, 
we can also introduce the 
inhomogeneous coordinate $p^{3}_{i}$ of fibers of $J(D,2)$ around $w$ with 
$\omega_{1}|_w={p^{3}_{1}}(w)\omega_{2}|_{w}+{p^{3}_{2}}(w)\pi |_{w}$.  
Moreover $w$ satisfies $d\varpi_{1}|_{w}\equiv d\varpi_{2}|_{w}\equiv 0$. 
However we have 
$d\varpi_{2}|_{w}\equiv \omega_{2}|_{w} \wedge \pi |_{w} \not =0$.    
Hence there does not exist an integral element.  
\end{proof}

Finally we consider the case of type (IV). 
\begin{lemma}\label{torsion-prolongation}
Let $(R,D)$ be a differential system of type $(${\rm IV}$)$ 
with $2$ independent and $1$ dependent 
variables. Then, the rank $2$ prolongation $\Sigma(R)$ is equal to $R$.   
\end{lemma}
\begin{proof}
We also use the covering  (\ref{Grassmann-cover}) of $\pi^{-1}(U)$ 
for the Grassmann bundle $J(D,2)$ 
and explicitly describe the defining equation of $\Sigma(R)$ 
in terms of the inhomogeneous Grassmann coordinate of fibers in $U_{\omega_{1}\omega_{2}}, U_{\omega_{1}\pi}, U_{\omega_{2}\pi }$. 
First we consider it on $U_{\omega_{1}\omega_{2}}$. For $w\in U_{\omega_{1}\omega_{2}}$, $w$ is a $2$-dimensional subspace of $D(v)$, 
where $p(w)=v$. Hence, by restricting $\pi$ to $w$, we can introduce the 
inhomogeneous coordinate $p^{1}_{i}$ of fibers of $J(D,2)$ around $w$ with 
$\pi |_{w}={p^{1}_{1}}(w)\omega_{1}|_{w}+{p^{1}_{2}}(w)\omega_{2}|_{w}$.  
Moreover $w$ satisfies $d\varpi_{1}|_{w}\equiv d\varpi_{2}|_{w}\equiv 0$
in (\ref{torsion-str-equ}). However we have 
$d\varpi_{1}|_{w}\equiv \omega_{1}|_{w}\wedge \omega_{2}|_{w}\not =0.$ 
Hence there does not exist an integral element. 
Next we consider on $U_{{\omega}_{1}\pi}$. 
For the same way, by restricting $\omega_{2}$ to $w$, we 
can also introduce the 
inhomogeneous coordinate $p^{2}_{i}$ of fibers of $J(D,2)$ around $w$ 
with $\omega_{2}|_{w}=p^{2}_{1}(w)\omega_{1}|_{w}+{p^{2}_{2}}(w)\pi |_{w}.$
Moreover $w$ satisfies 
$d\varpi_{1}|_{w}\equiv d\varpi_{2}|_{w}\equiv 0$. Thus we obtain  
\begin{align*}
d\varpi_{1}|_{w}&\equiv \omega_{1}|_{w}\wedge \omega_{2}|_{w},
                \equiv p^{2}_{2}(w)\omega_{1}|_{w}\wedge \pi|_{w},\\
d\varpi_{2}|_{w}&\equiv \omega_{2}|_{w}\wedge \pi|_{w},
                \equiv p^{2}_{1}(w)\omega_{1} |_{w}\wedge\pi |_{w}. 
\end{align*} 
In this way, we obtain the defining equations $f=g=0$ of $\Sigma(R)$ in $U_{\omega_{1}\pi}$ of $J(D,2)$, where 
$f=p_{2}^{2}, g=p_{1}^{2}$. Hence we have one trivial integral element. 
Finally, we consider on $U_{\omega_{2}\pi}$.  
For this situation, by restricting $\omega_1$ to $w$, we can also 
introduce the inhomogeneous coordinate $p^{3}_{i}$ of fibers of $J(D,2)$ around $w$ with  
$\omega_{1}|_w={p^{3}_{1}}(w)\omega_{2}|_{w}+{p^{3}_{2}}(w)\pi |_{w}$.  
Moreover $w$ satisfies 
$d\varpi_{1}|_{w}\equiv d\varpi_{2}|_{w}\equiv 0$. 
However, we have 
$d\varpi_{2}|_{w}\equiv \omega_{2}|_{w} \wedge \pi |_{w} \not =0$.    
Hence there does not exist an integral element. 
\end{proof}

Summarizing these lemmas in this section, 
we obtain the following theorem. 
\begin{theorem}\label{characterization-prolongation}
Let $(R,D)$ be a second-order regular overdetermined system 
of codimension $2$ for 
$2$ independent and $1$ dependent variables. 
Then the only involutive systems $($i.e. type of $(I))$ have non-trivial  
rank $2$ prolongations $\Sigma(R)$. Moreover, in this case, 
the rank $2$ prolongation $\Sigma(R)$ is a $S^1$-bundle over $R$. 
\end{theorem}
\begin{corollary}\label{prolongation-difference}
Let $(R,D)$ be a second-order 
regular overdetermined system of codimension $2$ for 
$2$ independent and $1$ dependent variables. Then we have 
\begin{align*}
&R^{(1)}=\Sigma(R) \Longleftrightarrow {\rm (II),\ (III)},\\ 
&R^{(1)}\not =\Sigma(R) \Longleftrightarrow {\rm (I),\ (IV)}.    
\end{align*}  
\end{corollary}

\section{Structures of rank 2 prolongations for involutive systems}
In this section, we study the geometric structures of rank 2 prolongations 
$(\Sigma(R), \hat D)$ of involutive systems $(R,D)$ with respect to 
2 independent and 1 dependent variables. 
For this purpose, we consider the decomposition  
\begin{equation}\label{geometric decomposition}
\Sigma(R)=\Sigma_{0}\cup \Sigma_{1},
\end{equation}
where 
$\Sigma_{i}=\left\{w \in \Sigma(R)\ |\ {\rm dim}\ 
(w\cap {\rm fiber})=i \right\}$ $(i=0, 1)$. 
Here {``\rm fiber''} means that the fiber of $TR\supset D\to TJ^1$. 
For the covering of the fibration 
$p:\Sigma(R)\to R$, we have 
\begin{equation*}
\Sigma_{0}|_{p^{-1}(U)}=U_{\omega_{1}\omega_{2}},\quad 
\Sigma_{1}|_{p^{-1}(U)}=U_{\omega_{1}\pi}
\backslash U_{\omega_{1}\omega_{2}}.
\end{equation*} 
The set $\Sigma_0$ is an open subset in $\Sigma(R)$, 
and $\Sigma_1$ is a codimension 1 submanifold in $\Sigma(R)$.\par   
Considering this decomposition, we obtain the following result.   
\begin{proposition}\label{symbol}
For any point $w\in \Sigma_{0}$, the symbol algebra $\mathfrak f^{0}(w)$ 
is isomorphic to  
$$\mathfrak f^{0}:=\mathfrak f_{-4}\oplus\mathfrak f_{-3}\oplus\mathfrak f_{-2}\oplus\mathfrak f_{-1}$$ 
whose bracket relations are given by 
$$[X_{a},\ X_{\omega_{2}}]=X_{\pi},\quad 
[X_{\pi},\ X_{\omega_{2}}]=X_{2},\quad     
[X_{1},\ X_{\omega_{1}}]=[X_{2},\ X_{\omega_{2}}]=X_{0},$$ 
and the other brackets are trivial. Here $\left\{X_{0},\ X_{1},\ X_{2},\ 
X_{\omega_{1}},\ X_{\omega_{2}},\ X_{\pi},\ X_{a}\right\}$ is a basis of $\mathfrak f^{0}$ and 
\begin{align*}
\mathfrak f_{-1}&=\left\{X_{\omega_{1}},\ X_{\omega_{2}},\ X_{a}\right\},\quad 
\mathfrak f_{-2}=\left\{X_{\pi}\right\},\quad
\mathfrak f_{-3}=\left\{X_{1},\ X_{2}\right\},\quad
\mathfrak f_{-4}=\left\{X_{0}\right\}.
\end{align*} 
For any point $w\in \Sigma_{1}$, the symbol algebra $\mathfrak f^{1}(w)$ 
is isomorphic to  
$$\mathfrak f^{1}:=\mathfrak f_{-4}\oplus\mathfrak f_{-3}\oplus\mathfrak f_{-2}\oplus\mathfrak f_{-1}$$ 
whose bracket relations are given by
$$[X_{b},\ X_{\pi}]=X_{\omega_{2}},\quad 
[X_{\pi},\ X_{\omega_{2}}]=X_{2},\quad    
[X_{1},\ X_{\omega_{1}}]=X_{0},$$ 
and the other brackets are trivial. Here  
$\left\{X_{0},\ X_{1},\ X_{2},\ 
X_{\omega_{1}},\ X_{\omega_{2}},\ X_{\pi},\ X_{b}\right\}$ 
is a basis of $\mathfrak f^{1}$ and 
\begin{align*}
\mathfrak f_{-1}&=\left\{X_{\omega_{1}},\ X_{\pi},\ X_{b}\right\},\quad
\mathfrak f_{-2}=\left\{X_{\omega_{2}}\right\},\quad
\mathfrak f_{-3}=\left\{X_{1},\ X_{2}\right\},\quad
\mathfrak f_{-4}=\left\{X_{0}\right\}.
\end{align*}
\end{proposition}
\begin{proof}
We first prove the assertion for the symbol algebras on $\Sigma_{0}$. 
We recall that the canonical system $\hat D$ on $\Sigma_{0}$ 
is given by 
$\hat D=\left\{\varpi_{0}=\varpi_{1}=\varpi_{2}=\varpi_{\pi}=0\right\},$ 
where $\varpi_{\pi}=\pi-a \omega_{2}.$ 
Then the structure equation of $\hat D$ on $\Sigma_{0}$ can be written as  
\begin{align}
d\varpi_{i}&\equiv 0 \hspace{3cm} \quad
\mod\ \varpi_{0},\ \varpi_{1},\ \varpi_{2},\ \varpi_{\pi}, \nonumber \\
d\varpi_{\pi}&\equiv \omega_{2}\wedge (da+f \omega_{1}) \hspace{2cm}  
\mod\ \varpi_{0},\ \varpi_{1},\ \varpi_{2},\ \varpi_{\pi}, \label{str-equ1}
\end{align} 
where $f$ is an appropriate function. Hence we have 
$\partial \hat D=\left\{\varpi_{0}=\varpi_{1}=\varpi_{2}=0\right\}=p_{*}^{-1}(D)$. 
The structure equation of $\partial \hat D$ is equal to 
the structure equation (\ref{involutive-str-equ}) of $(R,D)$. 
Now we provide a filtration structure 
$\left\{F^{p}\right\}_{p=-1}^{-4}$ on $\Sigma(R)$ 
around $w\in \Sigma_{0}$. We set 
$F^{-4}:=T\Sigma(R)$, 
$F^{-3}:=\left\{\varpi_{0}=0\right\}$,
$F^{-2}:=\partial \hat D=\left\{\varpi_{0}=\varpi_{1}=\varpi_{2}=0\right\}$, 
$F^{-1}:=\hat D$.
Moreover, for $w\in \Sigma_{0}$, we set 
$\mathfrak f_{-1}(w):=F^{-1}(w)=\hat D(w)$, 
$\mathfrak f_{-2}(w):=F^{-2}(w)/ F^{-1}(w)$, 
$\mathfrak f_{-3}(w):= F^{-3}(w)/ F^{-2}(w)$, 
$\mathfrak f_{-4}(w):= F^{-4}(w)/ F^{-3}(w)$, and  
$$\mathfrak f^{0}(w)=\mathfrak f_{-4}(w)\oplus \mathfrak f_{-3}(w)\oplus
\mathfrak f_{-2}(w)\oplus \mathfrak f_{-1}(w).$$ 
Then, by the definition of symbol algebras associated with 
filtration structures in Section 2, $\mathfrak f^{0}(w)$ has 
the structure of a nilpotent graded Lie algebra. 
We consider the bracket relation of $\mathfrak f^{0}(w)$. 
We take a coframe around $w\in \Sigma_{0}$ given by 
\begin{equation}
\left\{\varpi_{0},\ \varpi_{1},\ \varpi_{2},\ \varpi_{\pi},\  
\omega_{1},\ \omega_{2},\ \varpi_{a}:=da+f\omega_{1} \right\},  
\end{equation}
and the dual frame 
\begin{equation}
\left\{X_{0},\ X_{1},\ X_{2},\ X_{\pi},\ X_{\omega_{1}},\ 
X_{\omega_{2}},\ X_{a} \right\}. 
\end{equation} 
Then, the structure equations of each subbundle in the filtration 
$\left\{ F^{p}\right\}_{p=-1}^{-4}$
can be written by the coframe 
\begin{align*}
d\varpi_{i}&\equiv 0 \hspace{3cm} \quad
\mod\ \varpi_{0},\ \varpi_{1},\ \varpi_{2},\ \varpi_{\pi},\\
d\varpi_{\pi}&\equiv \omega_{2}\wedge \varpi_{a} \hspace{2cm}  
\mod\ \varpi_{0},\ \varpi_{1},\ \varpi_{2},\ \varpi_{\pi}, 
\end{align*}
\begin{align*}
d\varpi_{0}&\equiv \omega_1 \wedge \varpi_1+\omega_2 \wedge \varpi_2\quad \ 
{\rm mod\ } \varpi_0, \nonumber \\
d\varpi_{1}&\equiv \ 0 \quad \quad \quad \quad \quad \quad \ \ 
{\quad \quad \rm mod\ } 
\varpi_{0}, \varpi_{1}, \varpi_{2}, \\ 
d\varpi_{2}&\equiv \quad \quad \quad \quad \ \omega_2 \wedge \varpi_{\pi} 
{\quad \quad \rm mod\ } 
\varpi_{0}, \varpi_{1}, \varpi_{2}. \nonumber
\end{align*}
\begin{align*}
d\varpi_{0}\equiv \omega_{1}\wedge \varpi_{1}+\omega_{2}\wedge \varpi_{2}
  \ \mod\ &  \varpi_{0},\ \varpi_{1}\wedge \varpi_{2},\ \varpi_{1}\wedge \varpi_{\pi},\      
 \varpi_{2}\wedge \varpi_{\pi}. 
\end{align*}
We set 
$$[X_{\omega_{2}}, X_{a}]=A X_{\pi},\quad (A\in \mathbb R).$$
Then 
\begin{align*}
d\varpi_{\pi}(X_{\omega_{2}}, X_{a})&
=X_{\omega_{2}}\varpi_{\pi}(X_{a})-X_{a}\varpi_{\pi}(X_{\omega_{2}})
 -\varpi_{\pi}([X_{\omega_{2}}, X_{a}]),\\
 &=-\varpi_{\pi}([X_{\omega_{2}}, X_{a}])=-A.  
\end{align*}
On the other hand 
\begin{align*}
d\varpi_{\pi}(X_{\omega_{2}}, X_{a})&
=\omega_{2}(X_{\omega_{2}})\varpi_{a}(X_{a})
-\varpi_{a}(X_{\omega_{2}})\omega_{2}(X_{a}),\\
&=1.
\end{align*}
Therefore, $A=-1$. 
The other brackets are also obtained by the same argument and the definition 
of symbol algebras associated with the filtration structure. 
Thus we have the bracket relation of $\mathfrak f^{0}$.\par 
We next prove the assertion for the symbol algebras on $\Sigma_{1}$. 
We recall that $\Sigma_{1}$ is locally given by 
$U_{\omega_{1}\pi}\backslash U_{\omega_{1}\omega_{2}}$. 
Thus we may assume on 
$U_{\omega_{1}\pi}\backslash U_{\omega_{1}\omega_{2}}=\left\{b=0\right\}\subset U_{\omega_{1}\pi}$. 
Then the canonical system $\hat D$ is given by 
$\hat D=\left\{\varpi_{0}=\varpi_{1}=\varpi_{2}=\varpi_{\omega_{2}}=0\right\},$ 
where $\varpi_{\omega_{2}}=\omega_{2}-b \pi$. Note that 
$\varpi_{\omega_{2}}=\omega_{2}$ on $\Sigma_{1}$. 
The structure equation of $\hat D$ at a point $w\in \Sigma_{1}=\left\{b=0\right\}$ is 
\begin{align}
d\varpi_{i}&\equiv 0 \hspace{3cm} \quad
\mod\ \varpi_{0},\ \varpi_{1},\ \varpi_{2},\ \varpi_{\omega_{2}},\nonumber \\
d\varpi_{\omega_{2}}&\equiv \pi\wedge (db+f \omega_{1}) \hspace{2cm}  
\mod\ \varpi_{0},\ \varpi_{1},\ \varpi_{2},\ \varpi_{\omega_{2}}, \label{str-equ2}
\end{align} 
where $f$ is an appropriate function. Hence we have 
$\partial \hat D=\left\{\varpi_{0}=\varpi_{1}=\varpi_{2}=0\right\}=p_{*}^{-1}(D)$. 
The structure equation of $\partial \hat D$ is equal to 
the structure equation (\ref{involutive-str-equ}) of $(R,D)$. 
Here, we take the filtration 
which is same to the case of $\mathfrak f^{0}$. 
Then we have the symbol algebra $\mathfrak f^{1}(w)$ at a point $w\in \Sigma_{1}$ 
given by 
$$\mathfrak f^{1}(w)=\mathfrak f_{-4}(w)\oplus \mathfrak f_{-3}(w)\oplus
\mathfrak f_{-2}(w)\oplus \mathfrak f_{-1}(w).$$ 
We consider the bracket relation of $\mathfrak f^{1}(w)$. 
We take a coframe around $w\in \Sigma_{1}$ given by 
\begin{equation}
\left\{\varpi_{0},\ \varpi_{1},\ \varpi_{2},\ \varpi_{\omega_{2}},\  
\omega_{1},\ \pi,\ \varpi_{b}:=db+f\omega_{1} \right\},  
\end{equation}
and the dual frame 
\begin{equation}
\left\{X_{0},\ X_{1},\ X_{2},\ X_{\omega_{2}},\ X_{\omega_{1}},\ 
X_{\pi},\ X_{b} \right\}. 
\end{equation} 
Then the structure equations of each subbundle in the filtration 
$\left\{F^{p}\right\}_{p=-1}^{-4}$ 
can be written by the coframe 
\begin{align*}
d\varpi_{i}&\equiv 0 \hspace{3cm} \quad
\mod\ \varpi_{0},\ \varpi_{1},\ \varpi_{2},\ \varpi_{\omega_{2}},\\
d\varpi_{\omega_2}&\equiv \pi\wedge \varpi_{b} \hspace{2cm}  
\mod\ \varpi_{0},\ \varpi_{1},\ \varpi_{2},\ \varpi_{\omega_{2}}, 
\end{align*}
\begin{align*}
d\varpi_{0}&\equiv \omega_1 \wedge \varpi_1+\omega_2 \wedge \varpi_2\quad \ 
{\rm mod\ } \varpi_0, \nonumber \\
d\varpi_{1}&\equiv \ 0 \quad \quad \quad \quad \quad \quad \ \ 
{\quad \quad \rm mod\ } 
\varpi_{0}, \varpi_{1}, \varpi_{2}, \\ 
d\varpi_{2}&\equiv \quad \quad \quad \quad \ \omega_2 \wedge \pi  
{\quad \quad \rm mod\ } 
\varpi_{0}, \varpi_{1}, \varpi_{2}. \nonumber
\end{align*}
\begin{align*}
d\varpi_{0}\equiv \omega_{1}\wedge \varpi_{1} 
  \ \mod\ &  \varpi_{0},\ \varpi_{1}\wedge \varpi_{2},\ \varpi_{1}\wedge \varpi_{\omega_{2}},\      
 \varpi_{2}\wedge \varpi_{\omega_{2}}. 
\end{align*}
By the definition of the symbol algebras associated with the filtration structure and 
same argument in the proof of $\mathfrak f^{0}$, 
we obtain the bracket relation for $\mathfrak f^{1}$.     
\end{proof}
In the rest of this section, we mention a tower structure constructed  
by successive rank 2 prolongations of involutive systems. 
\begin{theorem}\label{tower}
Let $(R,D)$ be a regularly involutive system with $2$ independent $1$ dependent variables. 
Then the k-th rank $2$ prolongation $(\Sigma^{k}(R), \hat{D}^{k})$ of $(R,D)$ 
is also $S^1$-bundle over $\Sigma^{k-1}(R)$. 
\end{theorem}
\begin{proof}
From the expressions (\ref{str-equ1}) or (\ref{str-equ2}) of the structure equations of $(R,D)$, 
we easily show that we can define the 
$k$-th rank 2 prolongation $(\Sigma^{k}(R), \hat{D}^{k})$ 
successively. Then we have the assertion 
by using the same argument in the proof 
of Lemma \ref{involutive-prolongation}   
successively. 
\end{proof}

\section{Geometric singular solutions of involutive systems} 
In this section, we investigate the 
geometric singular solutions of involutive systems 
with 2 independent and 1 dependent variables. 
We first define the notion of 
the geometric singular solutions for regular PDEs (\cite{NS1, NS2}).  
\begin{definition}\label{solution}
Let $(R,D)$ be a second-order regular PDE 
in $J^{2}(\mathbb R^2,\mathbb R)$. 
For a $2$-dimensional integral manifold $S$ of $R$, 
if the restriction $\pi^{2}_{1}|_{R}: R\to J^1(\mathbb R^2,\mathbb R)$ 
of the natural projection $\pi^{2}_{1}:J^2\to J^1$ is an immersion 
on an open dense subset in $S$,  
then we call $S$  a {\it geometric solution} of $(R, D)$. If all points of 
geometric solutions $S$ are immersion points, 
then we call $S$ {\it regular solutions}. 
On the other hand, if geometric solutions $S$ have a 
nonimmersion point, then we call $S$ {\it singular solutions}.  
\end{definition}
From the definition, the image 
$\pi^{2}_{1}(S)$ of a geometric solution $S$ by the 
projection $\pi^{2}_{1}$ is Legendrian in $J^1(\mathbb R^2,\mathbb R)$,  
($\varpi_{0}|_{\pi^{2}_{1}(S)}=d\varpi_{0}|_{\pi^{2}_{1}(S)}=0$). 
From the proof of Lemmas 
\ref{par1-prolongation} and \ref{par2-prolongation}, 
there does not exist singular solutions of equations of types (II) and (III). 
On the other hand, 
Lemma \ref{torsion-prolongation} says the possibility of the existence of  
singular solutions for torsion type (IV). Of course, 
there does not exist regular solutions for these equations 
of torsion type. From now on, we investigate only involutive systems.\par   
Let $(R,D)$ be a regularly involutive system of codimension 2  
with 2 independent and 1 dependent variables. For this system $(R,D)$, we 
give two methods of the construction of geometric singular solutions 
which are given by 
the following. 
\begin{enumerate}
\item[(i)] We construct singular solutions of $(R,D)$ by using solutions of special type 
of rank 2 prolongations $(\Sigma(R),\hat D)$. 
\item[(ii)] We construct singular solutions   
of $(R,D)$ in terms of solutions (integral curves) of special type 
of rank 2 differential system $D_{B}$ on a 5-dimensional manifold $B$.  
\end{enumerate} 
The approach (i) is applicable to PDEs except involutive systems \cite{NS2}.  
On the other hand, the approach (ii) is a method specialized for 
involutive systems.\par 
We first mention the principle of the approach (i). 
We recall the decomposition (\ref{geometric decomposition}) of 
the rank 2 prolongations $(\Sigma(R),\hat D)$ of $(R,D)$. Here  
$(U_{\omega_{1}\omega_{2}}, \hat D)$ is the rank 2 prolongation 
with the independence condition $\omega_{1}\wedge \omega_{2}\not =0$. 
In general, for given second order regular overdetermined system 
$R=\left\{F=G=0\right\}$ with independent 
variables $x,y$, this prolongation corresponds to a 
third order PDE system which is obtained by partial derivation of $F=G=0$ 
for the two variables $x,y$. If we construct a solution of the system 
$(U_{\omega_{1}\omega_{2}}, \hat D)$, this solution $S$ is regular  
by the definition of $\Sigma_{0}$. On the other hand, 
if we construct a solution $S$ of $(\Sigma(R),\hat D)$ passing 
through $\Sigma_{1}=\left\{b=0\right\}\subset U_{\omega_{1}\pi}$, 
this solution $S$ is a singular solution of $(R,D)$ 
from the decomposition (\ref{geometric decomposition}). 
Thus, our strategy of this case is to find such a solution $S$ 
of $(\Sigma(R),\hat D)$. \par
We next mention the principle of the approach (ii). 
In fact, E. Cartan \cite{C1} 
characterized the overdetermined involutive systems $R$ 
by the condition that $R$ admits a 1-dimensional Cauchy characteristic system. Here, the Cauchy characteristic system $Ch(D)$ of a differential system $D$ on $R$ is defined by 
$$Ch(D)(x):=\left\{X\in D(x)\ |\ X\rfloor d\varpi_{i}\equiv 0\quad 
({\rm mod}\ \varpi_{0},\ \varpi_{1},\ \varpi_{2}) 
\quad {\rm for\ } i=0,1,2 \right\},$$
where, $\rfloor$ denotes the interior product (i.e., 
$X\rfloor d\varpi(Y)=d\varpi(X,Y)$.), and \\
$D=\left\{\varpi_{0}=\varpi_{1}=\varpi_{2}=0\right\}$ is defined locally by 
defining 1-forms $\left\{\varpi_{0},\ \varpi_{1},\ \varpi_{2}\right\}.$ 
From the expression (\ref{involutive-str-equ}) of the structure equation of an involutive system 
$(R,D)$, we obtain 
$Ch(D)=\left\{\varpi_{0}=\varpi_{1}=\varpi_{2}=\omega_{2}=\pi=0\right\}.$ 
This system $Ch(D)$ gives a 1-dimensional foliation. 
Hence, a leaf space $B:=R/Ch(D)$ is locally a 5-dimensional manifold. 
For this fibration $\pi^{R}_{B}:R\to B$, 
it is well-known that there exists a rank 2 differential system $D_{B}$ 
on the quotient space $B$ (\cite{C1}, \cite{Sat}, \cite{Y6}). 
Hence, if we construct integral curves of rank 2 differential system $(B,D_{B})$, we 
obtain integral surfaces $S$ of $(R,D)$ by using the fibration 
$\pi^{R}_{B}:R\to B$. Our strategy 
of this case is to find a singular solution among 
solutions obtained from such a technique. This principle is nothing but the 
theory of characteristic system. Namely, this approach is a theory of reduction 
into ordinary differential equations.\par  
As an application of the above discussion, we construct singular solutions of a typical equation based on the above two approach in the following subsection.     

\subsection{Singular solutions of Cartan's overdetermined system} 
We consider Cartan's overdetermined system 
$$R=\left\{r=\frac{t^3}{3},\ s=\frac{t^2}{2}\right\}.$$
The canonical system $D$ on $R$ is given by 
\begin{equation*}
\varpi_0= dz-pdx-qdy,\quad
\varpi_1= dp-\frac{t^3}{3}dx-\frac{t^2}{2}dy,\quad
\varpi_2= dq-\frac{t^2}{2}dx-tdy, 
\end{equation*}  
and the structure equation of $D$ is given by  
\begin{align}
d\varpi_{0}&\equiv dx\wedge dp+dy\wedge dq, 
\quad \quad {\rm mod}\ \varpi_{0} \nonumber \\
d\varpi_{1}&\equiv -t^2dt\wedge dx-tdt\wedge dy, \ 
{\quad \quad \rm mod\ } 
\varpi_{0}, \varpi_{1}, \varpi_{2},\\ 
d\varpi_{2}&\equiv -tdt\wedge dx-dt\wedge dy,
{\quad \quad \ \rm mod\ } 
\varpi_{0}, \varpi_{1}, \varpi_{2}. \nonumber
\end{align}
We take a new coframe 
$$\left\{\varpi_0,\ \hat\varpi_1:=\varpi_{1}-t\varpi_{2},\ \varpi_2, \pi:=dt,\ 
\omega_1:=dx,\ \omega_2:=tdx+dy\right\}.$$  
For this coframe, the above structure equation is rewritten as   
\begin{align}
d\varpi_{0}&\equiv \omega_{1}\wedge \hat\varpi_{1}+\omega_{2}\wedge \varpi_{2},
\quad \quad {\rm mod}\ \varpi_{0}, \nonumber \\
d\hat\varpi_{1}&\equiv  \ 0 \ \quad \quad \quad \quad \quad \quad
{\quad \quad \rm mod\ } 
\varpi_{0}, \hat\varpi_{1}, \varpi_{2},\label{Cartan-structure-equation} \\ 
d\varpi_{2}&\equiv \quad \quad \quad \quad \ 
 \omega_{2}\wedge \pi, 
{\quad \quad \ \rm mod\ } 
\varpi_{0}, \hat\varpi_{1}, \varpi_{2}. \nonumber
\end{align}
Hence this system $(R,D)$ is locally involutive.\par  
We first construct singular solutions of $(R,D)$ by using the approach (i). 
For this purpose, we need to prepare the rank 2 prolongation 
$(\Sigma(R), \hat D)$ of $(R,D)$ in terms of the Grassmann bundle 
$\pi:J(D,2)\to R$. For any open set $U\subset R$, $\pi^{-1}(U)$ is covered 
by 3 open sets in $J(D,2)$ such that 
$\pi^{-1}(U)=U_{xy}\cup U_{xt} \cup U_{yt},$ 
where 
\begin{align*}
U_{xy}:&=\left\{w\in \pi^{-1}(U)\ |\ dx|_{w}\wedge dy|_{w}\not=0 \right\},\quad 
U_{xt}:=\left\{w\in \pi^{-1}(U)\ |\ dx|_{w}\wedge dt|_{w}\not=0 \right\},\\
U_{yt}:&=\left\{w\in \pi^{-1}(U)\ |\ dy|_{w}\wedge dt|_{w}\not=0 \right\}. 
\end{align*}
We next explicitly describe the defining equation of $\Sigma(R)$ 
in terms of the inhomogeneous Grassmann coordinate of fibers in $U_{xy}, U_{xt}, U_{yt}$. 
First, we consider it in $U_{xy}$. 
For $w\in U_{xy}$, $w$ is a $2$-dimensional subspace of $D(v)$, 
where $p(w)=v$. 
Hence, by restricting $dt$ to $w$, we can introduce the 
inhomogeneous coordinate $p^{1}_{i}$ of fibers of $J(D,2)$ around $w$ with  
$dt|_{w}=p_{1}^{1}(w)dx|_{w}+p_{2}^{1}(w)dy|_{w}$. 
Moreover $w$ satisfies $d\varpi_{1}|_{w}\equiv d\varpi_{2}|_{w}\equiv 0$. 
Thus we have 
\begin{align*}
d\varpi_{1}|_{w}&\equiv 
\left(t^{2}p_{2}^{1}(w)-tp_{1}^{1}(w)\right) dx|_{w}\wedge dy|_{w}, \\ 
d\varpi_{2}|_{w}&\equiv 
\left(t p_{2}^{1}(w)-p_{1}^{1}(w)\right) dx|_{w}\wedge dy|_{w}.  
\end{align*} 
In this way, we obtain the defining equations $f=0$ of 
$\Sigma(R)$ in $U_{xy}$ of $J(D,2)$, where 
$f=p_{1}^{1}-tp_{2}^{1}$ Then $df$ does not vanish on $\left\{f=0\right\}$. 
Next we consider in $U_{xt}$. 
For $w\in U_{xt}$, $w$ is a $2$-dimensional subspace of $D(v)$, 
where $p(w)=v$. 
Hence, by restricting $dy$ to $w$, we can introduce the 
inhomogeneous coordinate $p^{2}_{i}$ of fibers of $J(D,2)$ around $w$ with  $dy|_{w}={p^{2}_{1}}(w)dx|_{w}+{p^{2}_{2}}(w)dt|_{w}$. 
Moreover $w$ satisfies $d\varpi_{1}|_{w}\equiv d\varpi_{2}|_{w}\equiv 0$. 
In this situation, 
it is sufficient to consider the condition 
$d\varpi_{2}|_{w}(\equiv 
\left( t+p_{1}^{2}(w)\right) dx|_{w}\wedge dt|_{w})\equiv 0$. 
Then, for the defining function $f=t+p_{1}^{2}$ of $\Sigma(R)$, 
$df$ does not vanish 
on $\Sigma(R)$. Finally, we consider in $U_{yt}$. 
For $w\in U_{yt}$, $w$ is a $2$-dimensional subspace of $D(v)$, where 
$p(w)=v$. Hence, by restricting $dx$ to $w$, we can introduce the 
inhomogeneous coordinate $p^{3}_{i}$ of fibers of 
$J(D,2)$ around $w$ with 
$dx|_{w}={p^{3}_{1}}(w)dy|_{w}+{p^{3}_{2}}(w)dt|_{w}$.  
Moreover $w$ satisfies $d\varpi_{1}|_{w}\equiv d\varpi_{2}|_{w}\equiv 0$. 
Here, $d\varpi_{2}|_{w}\equiv \left( 1+tp_{1}^{3}(w)\right) 
dy|_{w}\wedge dt|_{w}$. Then, for the defining function $f=1+tp_{1}^{3}$ of $\Sigma(R)$, $df$ does not vanish 
on $\Sigma(R)$. Therefore, we have the covering for the fibration $p:\Sigma(R)\to R$ such that $p^{-1}(U)=U_{xy}\cup U_{xt} \cup U_{yt}$.
However this covering is not essential. 
\begin{proposition}
Let $R$ be Cartan's overdetermined system and 
$U$ an open set on $R$. Then we have 
\begin{equation}\label{Cartan cover}
p^{-1}(U)=U_{xy}\cup U_{xt}.
\end{equation} 
\end{proposition}
\begin{proof}
We show that $U_{yt}\subset U_{xt}$. Let $w$ be any point in $U_{yt}$. Here,  
if $w\not \in U_{xt}$, we have 
$$dx|_{w}\wedge dt|_{w}=-\frac{1}{t}(w)dy|_{w}\wedge dt|_{w}.$$ 
Hence we have the condition $1/t(w)=0$. However there does not exist 
such a point $w$. 
\end{proof}
We have the following description of the canonical system 
$\hat D$ of rank $3$: 
For $U_{xy}$,  
$\hat D=\left\{\varpi_{0}=\varpi_{1}=\varpi_{2}=\varpi_{t}=0\right\},$ 
where $\varpi_{t}=dt-ta dx-ady$ and $a$ is a fiber coordinate. 
For $U_{xt}$,  
$\hat D=\left\{\varpi_{0}=\varpi_{1}=\varpi_{2}=\varpi_{y}=0\right\},$ 
where $\varpi_{y}=dy+tdx-bdt$ and $b$ 
is a fiber coordinate. The decomposition $\Sigma(R)=\Sigma_{0}\cup\Sigma_{1}$ is given by $\Sigma_{0}|_{p^{-1}(U)}=U_{xy}$, 
$\Sigma_{1}|_{p^{-1}(U)}=U_{xt} \backslash U_{xy}$, respectively.\par 
By using the approach (i), we construct the geometric singular solutions 
of $(\Sigma(R), \hat D)$ passing through $\Sigma_{1}$.   
Let $\iota:S \hookrightarrow U_{xt}$ 
be a graph defined by 
$$(x,\ y(x,t),\ z(x,t),\ p(x,t),\ q(x,t),\ t,\ b(x,t)).$$  
If $S$ is an integral submanifold of $(U_{xt}, \hat D)$, 
then the following conditions are satisfied:  
\begin{align}
\iota^{*}\varpi_{0}=&(z_x-p-qy_x)dx+(z_t-qy_t)dt=0, \label{1CartanODS1} \\ 
\iota^{*}\varpi_{1}=&\left( p_x-\frac{t^3}{3}-\frac{t^2}{2}y_{x} \right)dx + 
                    \left( p_{t}-\frac{t^2}{2}y_{t} \right) dt=0, \label{1CartanODS2} \\
\iota^{*}\varpi_{2}=&\left (q_x-\frac{t^2}{2}-ty_x \right) dx+ (q_t-ty_t)dt=0, \label{1CartanODS3}\\
\iota^{*}\varpi_{y}=&(y_{x}+t)dx+(y_t-b)dt=0. \label{1CartanODS4} 
\end{align}
From these conditions, we have 
\begin{align}
& z_x-p+qt=0,\ z_t-bq=0, \label{2CartanODS1} \\
& p_x+\frac{t^3}{6}=0,\  p_{t}-\frac{bt^2}{2}=0, \label{2CartanODS2} \\ 
& q_x+\frac{t^2}{2}=0,\  q_{t}-bt=0, \label{2CartanODS3}\\  
& y_{x}+t=0,\ y_t-b=0. \label{2CartanODS4} 
\end{align}
We have $y=-tx+y_{0}(t)$ from $(\ref{2CartanODS4})$. Note that 
the condition passing through $\Sigma_{1}$ is $y_t(0,0)=y^{\prime}_{0}(0)=0$. 
From $(\ref{2CartanODS3})$, we have 
$q=-t^2x/2+ty_{0}(t)-\int y_{0}(t)dt$. 
From $(\ref{2CartanODS2})$, we have 
$p=-t^3x/6+t^{2}y_{0}(t)/2-\int t y_{0}(t)dt$. 
From $(\ref{2CartanODS1})$, we have 
\begin{equation*}
\frac{x^{2}t^{3}}{6}-x\left\{ \frac{t^2y_{0}(t)}{2}+\int ty_{0}(t)dt-t\int y_{0}(t)dt \right\}
+\frac{t y_{0}^{2}(t)}{2}+\frac{1}{2}\int y_{0}^{2}(t)dt-y_{0}(t)\int y_{0}(t)dt. 
\end{equation*}
Consequently, we obtain the singular solutions of the form  
\begin{align}
( &x,\quad -xt+y_{0}(t), \nonumber \\
& \frac{x^{2}t^{3}}{6}-x\left\{ \frac{t^2y_{0}(t)}{2}+\int ty_{0}(t)dt-t\int y_{0}(t)dt \right\}
+\frac{ty_{0}^{2}(t)}{2}+\frac{1}{2}\int y_{0}^{2}(t)dt-y_{0}(t)\int y_{0}(t)dt, \label{Cartan-singular1}  \\
& -t^3x/6+t^{2}y_{0}(t)/2-\int t y_{0}(t)dt,\quad -t^2x/2+ty_{0}(t)-\int y_{0}(t)dt,\quad t,\quad -x+y^{\prime}_{0}(t)), \nonumber 
\end{align}
where $y_{0}(t)$ is a function on $S$ depending only $t$ 
and which satisfies $y^{\prime}_{0}(0)=0$. 
From this condition $y^{\prime}_{0}(0)=0$, these solutions 
have singularities at the origin in $J^1(\mathbb R^2, \mathbb R)$. 

We next construct the singular solutions by using the approach (ii). 
From the structure equation (\ref{Cartan-structure-equation}), we have 
\begin{align*}
Ch(D)&=\left\{\varpi_{0}=\hat\varpi_{1}=\varpi_{2}
=\omega_{2}=\pi=0\right\},\\
&={\rm span}\left\{\frac{\partial}{\partial x}-t\frac{\partial}{\partial y}
+(p-tq)\frac{\partial}{\partial z}-\frac{t^3}{6}\frac{\partial}{\partial p}
-\frac{t^2}{2}\frac{\partial}{\partial q}\right\}.
\end{align*}   
Hence we have a local coordinate $(x_1,x_2,x_3,x_4,x_5)$ on the 
leaf space $B:=R/Ch(D)$ given by   
\begin{align*} 
x_1:&=z-xp+xqt+\frac{1}{6}x^2t^3,\quad 
x_2:=p-qt+\frac{1}{2}yt^2+\frac{1}{6}t^3x,\\
x_3:&=-q+\frac{1}{2}yt,\quad x_4:=y+xt,\quad x_5:=-t.
\end{align*}  
Conversely, $R$ is locally a $\mathbb R$-bundle on $B$. 
If we take a coordinate function $\lambda$ of the fiber $\mathbb R$,   
then the coordinate $(x,y,z,p,q,t)$ is expressed in  
terms of the coordinate \\ $(x_1,x_2,x_3,x_4,x_5,\lambda)$ defined by 
\begin{align} 
x&=\lambda,\quad y=x_4+\lambda x_5,\quad  
z=x_1+\lambda x_2-\frac{1}{2}\lambda x_4(x_5)^2-\frac{1}{6}\lambda^{2} (x_5)^3, \label{G_2fibration} \\
p&=x_2+x_3x_5+\frac{1}{6}\lambda (x_5)^3,\quad q=-x_3-\frac{1}{2}x_{4}x_{5}
-\frac{1}{2}\lambda (x_5)^2,\quad   
t=-x_5. \nonumber
\end{align} 
On the base space $B$, we consider a rank 2 differential system  
$D_B=\left\{\alpha_1=\alpha_2=\alpha_3=0\right\}$ given by 
\begin{align*}
\alpha_1 &=dx_1+\left(x_3+\frac{1}{2}x_4x_5\right) dx_4,\quad
\alpha_2 =dx_2+\left(x_3-\frac{1}{2}x_4x_5\right) dx_5,\\ 
\alpha_3 &=dx_3+\frac{1}{2}\left(x_4dx_5-x_5dx_4\right). 
\end{align*} 
It is well-known that this system $D_B$ is a flat model 
of (2,3,5)-distributions \cite{Y4}. Indeed, we can check this 
fact by calculating derived systems. Moreover, it is also 
well-known that $D_B$ has 
infinitesimal automorphism $G_2$ (\cite{C1},\ \cite{Y5}). 
For the projection 
$p:R \to B$, generator 1-forms of $D$ and $D_B$
are related as follows: 
\begin{equation*}
\varpi_0:=p^{*}\alpha_1+x p^{*}\alpha_ 2,\quad
\varpi_1:=p^{*}\alpha_2-xp^{*}\alpha_ 3,\quad
\varpi_2:=-p^{*}\alpha_3. 
\end{equation*}
Thus, $(B, D_B)$ is a retracting space 
of $(R, D)$, that is, $(B, D_B)=(p(R), p_{*}D)$. 
By using this correspondence, E. Cartan obtained the explicit 
description of regular solutions of $(R, D)$ 
which are constructed by solution 
curves of $(B,D_B)$ (cf. \cite{C1}, \cite{Sat}). 
In contrast to this result, we construct anew singular solutions in the following. 
We consider integral curves $c(\tau)$ of $D_B$ given by  
\begin{align}
x_1&=\int \left\{\varphi^{\prime}\int \varphi d\tau-\varphi\varphi^{\prime} \tau \right\}d\tau,\quad     
x_2=\int \left\{\int \varphi d\tau \right\}d\tau,\nonumber \\  
x_3&=-\frac{1}{2}\int(\varphi-\tau \varphi^{\prime})d\tau,\quad  x_4=\varphi(\tau),\quad x_5=\tau. \label{integral curve}
\end{align} 
where $\tau$ is a parameter of curves, and $\varphi(\tau)$ is 
an arbitrary smooth function of $\tau$. Here, we assume the condition 
$\varphi^{\prime}(0)=0$ to consider singular solutions which have singularities 
at the origin in $J^1(\mathbb R^2, \mathbb R)$. 
Then, from relations (\ref{G_2fibration}) and (\ref{integral curve}),   
we obtain the following singular solutions: 
\begin{align}
( &x,\quad -xt+\varphi(-t),\  
 \frac{x^{2}t^{3}}{6}-x\left\{ \frac{t^2\varphi(-t)}{2}
+\int t\varphi(-t)dt-t\int \varphi(-t)dt \right\}\nonumber \\ &
+\frac{t\varphi^{2}(-t)}{2}+\frac{1}{2}\int \varphi^{2}(-t)dt-\varphi(-t)\int \varphi(-t)dt,\\
& -\frac{t^3x}{6}+\frac{t^{2}\varphi(-t)}{2}-\int t \varphi(-t)dt,\quad 
-\frac{t^2x}{2}+t\varphi(-t)-\int \varphi(-t)dt,\quad t). \nonumber 
\end{align}
These singular solutions are equal to the singular solutions (\ref{Cartan-singular1}) 
obtained by the approach (i).  

{\bf Acknowledgment} \ 
The author would like to thank 
Kazuhiro Shibuya for helpful discussions. 
He also would like to thank Professor   
Keizo Yamaguchi for encouragement and useful advice.    
The author is also supported by 
Osaka City University University Advanced Mathematical 
Institute and the JSPS Institutional Program for 
Young Researcher Overseas Visits (visiting Utah-State University).

\end{document}